\documentclass[11pt,leqno]{amsart}
\usepackage{amsmath,amssymb,amsthm}
\usepackage{hyperref}

\usepackage{bbm}

\DeclareMathOperator{\co}{co}

\renewcommand{\geq}{\geqslant}
\renewcommand{\leq}{\leqslant}

\DeclareMathOperator{\spann}{span}

\newcommand{\Lip}{{\mathrm{Lip}}_0}

\newcommand{\aconv}{\mathop\mathrm{aconv}}

\newcommand{\ext}[1]{\operatorname{ext}\left(#1\right)}

\newtheorem{theorem}{Theorem}[section]
\newtheorem{lemma}[theorem]{Lemma}

\newtheorem{proposition}[theorem]{Proposition}

\theoremstyle{definition}
\newtheorem{definition}[theorem]{Definition}
\newtheorem{example}[theorem]{Example}

\theoremstyle{remark}
\newtheorem{remark}[theorem]{Remark}

\numberwithin{equation}{section}

\def\fnote#1{\footnote}

\def\ignora#1{}
\def\n3#1{\left\vert  \! \left\vert \! \left\vert \, #1 \, \right\vert \!
  \right\vert \! \right\vert }

\newcommand{\pten}{\ensuremath{\widehat{\otimes}_\pi}}

\begin{document}

\title{ Banach spaces which always produce octahedral spaces of operators }

\author{Abraham Rueda Zoca}
\address[A. Rueda Zoca]{Universidad de Granada, Facultad de Ciencias.
Departamento de An\'{a}lisis Matem\'{a}tico, 18071-Granada
(Spain)
	\newline
	\href{https://orcid.org/0000-0003-0718-1353}{ORCID: \texttt{0000-0003-0718-1353} }}
\email{\texttt{abrahamrueda@ugr.es}}
\urladdr{\url{https://arzenglish.wordpress.com}}

\thanks{The research of Abraham Rueda Zoca was supported by MCIN/AEI/10.13039/501100011033: Grants PID2021-122126NB-C31 and PID2021-122126NB-C32;  and by Junta de Andaluc\'ia: Grants FQM-0185 and PY20\_00255.}



\maketitle

\markboth{ABRAHAM RUEDA ZOCA}{BANACH SPACES WHICH ALWAYS PRODUCE OCTAHEDRAL ...}

\begin{abstract}
We characterise those Banach spaces $X$ which satisfy that $L(Y,X)$ is octahedral for every non-zero Banach space $Y$. They are those satisfying that, for every finite dimensional subspace $Z$, $\ell_\infty$ can be finitely-representable in a part of $X$ kind of $\ell_1$-orthogonal to $Z$. We also prove that $L(Y,X)$ is octahedral for every $Y$ if, and only if, $L(\ell_p^n,X)$ is octahedral for every $n\in\mathbb N$ and $1<p<\infty$. Finally, we find examples of Banach spaces satisfying the above conditions like $\Lip(M)$ spaces with octahedral norms or $L_1$-preduals with the Daugavet property.
\end{abstract}

\section{Introduction}

According to \cite[Remark~II.5.2]{god}, the norm of a Banach space $X$ is
\emph{octahedral} if, for every finite dimensional subspace $E$ of $X$
and every $\varepsilon > 0$, there exists $y \in S_X$ such that
\begin{equation*}
  \| x + \lambda y \| \ge (1-\varepsilon)(\|x\|+|\lambda|)
  \ \mbox{for every } x \in E\mbox{ and every }\lambda \in \mathbb{R}.
\end{equation*}

Octahedral norms were studied at the end of the eighties in succesive papers \cite{god,gk} because it turns out that such norms characterise the containment of $\ell_1$. Indeed, in \cite[Theorem II.4]{god} it is proved that a Banach space $X$ contains an isomorphic copy of $\ell_1$ if, and only if, $X$ can be equivalently renormed with an octahedral norm.

However, octahedral norms have received much more attention in the recent years because in \cite[Theorem 2.1]{blrjfa14} it is proved that $X$ is octahedral if, and only if, every convex combination of $w^*$-slices of $B_{X^*}$ has diameter two. Since then, octahedral norms and variations of such norms have been studied in many different contexts (see. e.g. \cite{amcrz2022,cll2021,hln2018,ll2019,loru2021}).

One of the areas where octahedrality has been intensively studied is in spaces of operators, that is, it has been analysed when the space of bounded operators $L(X,Y)$ between two Banach spaces $X$ and $Y$ is octahedral. The motivation for this interest comes from \cite[Question (b)]{aln13}, where the authors asked when the projective tensor product $X\pten Y$ satisfies that all the convex combination of slices of its unit ball have diameter two. Thanks to the duality $(X\pten Y)^*=L(X,Y^*)$ and the above mentioned \cite[Theorem 2.1]{blrjfa14}, the above question is equivalent to determining when the norm of the space of operators is  octahedral.

In \cite[Theorem 3.5]{blrope} it is proved that if $Y^*$ and $X$ are octahedral then $H$ is octahedral for any subspace $H\subseteq L(Y,X)$ containing finite-rank operators $F(Y,X)$. More examples of octahedral spaces of operators were given in \cite{hlp2}. It was shown, however, that the above is not the case if we remove octahedrality on $Y^*$, and in fact octahedrality of $L(Y,X)$ is connected with finite-representability of $Y$ in $X$. Indeed, in \cite[Lemma 3.7]{llr2} it is proved that if some subspace $H$ of $L(Y,X)$ is octahedral and $Y$ is uniformly convex then $Y$ is finitely representable in $X$.

The connection between finite-representability and octahedrality of spaces of operators have shown to be much deeper. Indeed, a kind of converse is established in \cite[Theorem 3.2]{laru2020} where it is proved that if $X$ is a Banach space which is finitely representable in $\ell_1$ and with the metric approximation property, then $L(X,Y)$ is octahedral if $Y$ is octahedral.

In this note we will focus on the following problem: which Banach spaces $X$ satisfies that $L(Y,X)$ is octahedral for every Banach space $Y$? We will refer to these spaces as \textit{universally octahedral} (see Definition \ref{defi:uniocta}). Observe that, in order to solve a problem about octahedrality of spaces of vector-valued Lipschitz functions, it is proved in \cite[Theorem 3.1]{laru2020} that $\Lip(M)$, the space of Lipschitz functions over $M$, is universally octahedral whenever $\Lip(M)$ is octahedral.

Anyway, in view of \cite[Lemma 3.7]{llr2} one should think that if $X$ is universally octahedral then it should be an octahedral space such that every uniformly convex Banach space is finitely representable in it. This intuition is confirmed in Lemma \ref{lemma:uconvexo}, where we observe that a necessary condition for universal octahedrality is that, roughly speaking, given any finite dimensional subspace $Z$ of $X$ and any $\varepsilon>0$, then any finite dimensional uniformly convex Banach space can be $(1+\varepsilon)$-embedded in the space of ``$\varepsilon$-orthogonal vectors to $Z$''. In a further step, making use of approximations in Banach-Mazur distance, we obtain in Theorem \ref{theo:condinece} that we can replace in the above statement uniformly convex Banach spaces with $\ell_\infty^n$ for every $n$. More precisely, we prove that if $X$ is universally octahedral then, given any finite dimensional subspace $Z$ of $X$, any $n\in\mathbb N$ and any $\varepsilon>0$, we can find a norm-one operator $\Phi:\ell_\infty^n\longrightarrow X$ such that
$$\Vert z+\Phi(y)\Vert\geq (1-\varepsilon)(\Vert z\Vert+\Vert y\Vert)$$
holds for every $z\in Z$ and every $y\in \ell_\infty$.

The converse, making use of the finite-representability of every Banach space in $c_0$ together with the celebrated characterisation of $L_1$-preduals due to J. Lindenstrauss \cite[Theorem 6.1]{linds}, is proved in Theorem \ref{theo:condisufi}. As a consequence we obtain, in Theorem \ref{theo:carauniocta}, that a Banach space $X$ is universally octahedral if, and only if, $L(\ell_p^n,X)$ is octahedral for every $1\leq p\leq \infty$ and every $n\in\mathbb N$, which is in turn equivalent to the condition that, given any finite dimensional subspace $Z$ of $X$, any $n\in\mathbb N$ and any $\varepsilon>0$, we can find a norm-one operator $\Phi:\ell_\infty^n\longrightarrow X$ such that
$$\Vert z+\Phi(y)\Vert\geq (1-\varepsilon)(\Vert z\Vert+\Vert y\Vert)$$
holds for every $z\in Z$ and every $y\in \ell_\infty^n$.

Observe that the above condition is strictly stronger than the mere finite-representability of $\ell_\infty$ in $X$. Indeed, in Example \ref{exam:nounivoctalinf} we construct an example of a octahedral Banach space which contains an isomorphic copy of $\ell_\infty$ but failing the universal octahedrality. This shows that, in order to obtain universal octahedrality, the requierement that the copies of $\ell_\infty^n$ can be found in the orthogonal part of any finite dimensional subspace can not be relaxed.

Another relevant example is given in Example \ref{example:uninoc0octa}, where it is shown that a universally octahedral space does not have to contain $c_0$ isomorphically.

In Section \ref{section:examples} we aim to find new examples of Banach spaces $X$ which are universally octahedral. We begin by observing that a sufficient condition for universal octahedrality of a $X$ is the following: for every finite dimensional subspace $Z$ of $X$ and every $\varepsilon>0$, there exists a subspace $Y$ of $X$ which is isometrically isomorphic to $c_0$ and such that
$$\Vert z+y\Vert\geq (1-\varepsilon)(\Vert z\Vert+\Vert y\Vert)$$
holds for every $z\in Z$ and every $y\in Y$ (we define this property in Definition \ref{defi:c0octa} as \textit{$c_0$-octahedral}). The reason to introduce this definition is double. The first one is to recover the technique followed in \cite[Theorem 3.1]{laru2020}, where it is proved that $\Lip(M)$ is universally octahedral when it is octahedral, but whose proof is based on \cite[Lemma 3.3]{laru2020}, where it is preciselly proved that $\Lip(M)$ is $c_0$-octahedral. On the other hand, in spite of the fact that Example \ref{example:uninoc0octa} shows that universal octahedrality does not imply $c_0$-octahedrality, there is a strong connection through ultrapower spaces. In fact, in Proposition \ref{prop:ultrapowers} it is proved that $X$ is universally octahedral if, and only if, $X_\mathcal U$ is $c_0$-octahedral for every free ultrafilter $\mathcal U$ over $\mathbb N$. We end the paper with Theorem \ref{theo:L1preduals}, where we prove that every $L_1$-predual which is octahedral is indeed universally octahedral, using recent tools developed in \cite{mr2022}.

\section{Notation and preliminary results}

We will consider real Banach spaces. Given a Banach space $X$, we will denote the closed unit ball and the unit sphere of $X$ by $B_X$ and $S_X$ respectively. We will also denote by $X^*$ the topological dual of $X$. Given two Banach spaces $X$ and $Y$ denote by $L(X,Y)$ (respectively $F(X,Y)$) the space of linear bounded operators (respectively the finite-rank operators) from $X$ to $Y$. 

According to \cite[Definition 11.1.1]{alka}, given two Banach spaces $X$ and $Y$, we say that \textit{$X$ is finitely representable in $Y$} if, given any finite dimensional subspace $E$ of $X$ and any $\varepsilon>0$, there exist a subspace $F$ of $Y$  and a linear continuous bijection $T:E\longrightarrow F$ such that $\Vert T\Vert \Vert T^{-1}\Vert\leq 1+\varepsilon$. This notion encodes the idea that $Y$ contains all the finite dimensional structure of $X$. Observe that every Banach space is finitely-representable in $c_0$ \cite[Example 11.1.2]{alka}. Moreover, as a consequence of the Principle of Local Reflexivity (c.f. e.g. \cite[Lemma 9.15]{checos}), for every Banach space $X$ it follows that $X^{**}$ is finitely representable in $X$. We refer the interested reader to \cite[Chapter 11]{alka} and references therein for background about finite representability of Banach spaces. 

Let us include here, for easy reference, the following lemma, which is extracted from \cite[Lemma 11.1.11]{alka}.

\begin{lemma}\label{lema:caraepsisometrynets}
Let $E$ be a finite dimensional Banach space and let $\{x_j: 1\leq j\leq N\}\subseteq S_X$ be an $\varepsilon$-net of $S_E$. Let $T:E\longrightarrow X$ be a linear mapping such that $(1-\varepsilon)\leq \Vert T(x_j)\Vert\leq (1+\varepsilon)$ holds for every $1\leq j\leq N$. Then, for every $e\in E$, we have
$$\left(\frac{1-3\varepsilon}{1-\varepsilon} \right)\Vert e\Vert\leq \Vert T(e)\Vert\leq \left(\frac{1+\varepsilon}{1-\varepsilon} \right)\Vert e\Vert.$$
\end{lemma}

Given a sequence of Banach spaces $\{X_n:n\in\mathbb N\}$ we denote 
$$\ell_\infty(\mathbb N,X_n):=\left\{f\colon \mathbb N\longrightarrow \prod\limits_{n\in \mathbb N} X_n: f(n)\in X_n\ \forall n\text{ and }\sup_{n\in \mathbb N}\Vert f(n)\Vert<\infty\right\}.$$
Given a non-principal ultrafilter $\mathcal U$ over $\mathbb N$, consider $c_{0,\mathcal U}(\mathbb N,X_n):=\{f\in \ell_\infty(\mathbb N,X_n): \lim_\mathcal U \Vert f(n)\Vert=0\}$. The \textit{ultrapower of $\{X_n:n\in\mathbb N\}$ with respect to $\mathcal U$} is
the Banach space
$$(X_n)_\mathcal U:=\ell_\infty(\mathbb N,X_n)/c_{0,\mathcal U}(\mathbb N,X_n).$$
We will naturally identify a bounded function $f\colon\mathbb N\longrightarrow \prod\limits_{n\in \mathbb N} X_n$ with the element $(f(n))_{n\in\mathbb N}$. In this way, we denote by $(x_n)_\mathcal U$ or simply by $(x_n)$, if no confusion is possible, the coset in $(X_n)_\mathcal U$ given by $(x_n)_{n\in \mathbb N}+c_{0,\mathcal U}(\mathbb N,(X_n))$.

From the definition of the quotient norm, it is not difficult to prove that $\Vert (x_n)_\mathcal U\Vert=\lim_\mathcal U \Vert x_n\Vert$ holds for every $(x_n)\in (X_n)_\mathcal U$.

%

When $X_n=X$ holds for every $n\in\mathbb N$, the definition of the norm on $X_\mathcal U$ yields a canonical inclusion $j:X\longrightarrow X_\mathcal U$ given by the equation
$$j(x):=(x)_\mathcal U.$$
This inclusion is an into linear isometry, so $X$ can be isometrically embedded in $X_\mathcal U$. Moreover, $X_\mathcal U$ is finitely representable in $X$ \cite[Proposition 11.1.12]{alka}.

Given a Banach space $X$ we say that $X$ is an \textit{$L_1$-predual} if $X^*=L_1(\mu)$ isometrically for some measure $\mu$. Let us include here for easy reference in the text the following result.

\begin{theorem}\label{theo:lindenstrauss}\cite[Theorem 6.1]{linds}
Let $X$ be a Banach space. The following assertions are equivalent:
\begin{enumerate}
\item $X$ is an $L_1$-predual.
\item Every compact operator $T:Y\longrightarrow X$ has, for every $\varepsilon>0$ and every Banach space $Z$ containing $Y$, an extension $\hat T:Z\longrightarrow X$ such that $\Vert \hat T\Vert\leq (1+\varepsilon)\Vert T\Vert$. 
\end{enumerate}
\end{theorem}

Strongly related to $L_1(\mu)$-spaces are the $L$-summands. A projection $P\colon X\longrightarrow X$ on a Banach space $X$ is said to be an \emph{$L$-projection} if $\|x\|=\|Px\|+\|x-Px\|$ for every $x\in X$. The range of an $L$-projection is called an \emph{$L$-summand}. We refer the reader to \cite{hww} for a vast background about $L$-summands.

\section{Characterisation of universally octahedral spaces}\label{section:charauniocta}

Let us start with the main definition of the paper.

\begin{definition}\label{defi:uniocta}
Let $X$ be a Banach space. We will say that $X$ is \textit{universally octahedral} if $L(Y,X)$ is octahedral for every non-zero Banach space $Y$.
\end{definition}

The aim of this section is to provide a characterisation of universally octahedral Banach spaces. In order to do so, let us start with the following preliminary lemma, which is a strengthening of \cite[Lemma 3.7]{llr2}. Recall that a Banach space $X$ is said to be \textit{uniformly convex} if, for every $\varepsilon>0$, there exists $\delta(\varepsilon)>0$ such that 
$$\left.\begin{array}{c}
x,y\in B_X\\
\Vert x+y\Vert>2-\delta(\varepsilon)
\end{array} \right\}\Rightarrow \Vert x-y\Vert<\varepsilon.$$
Examples of uniformly convex Banach spaces are $L_p(\mu)$ for $1<p<\infty$ thanks to Clarkson inequality (see \cite[Chapter 9]{checos} for background about uniform convexity).

\begin{lemma}\label{lemma:uconvexo}
Let $X$ be a Banach space and $Y$ be a finite dimensional uniformly convex Banach space. Assume that $L(Y,X)$ is octahedral. Then, for every $\varepsilon>0$ and for every finite dimensional subspace $Z$ of $X$, there exists an element $T\in B_{L(Y,X)}$ such that
$$\Vert z+T(y)\Vert\geq (1-\varepsilon)^2(\Vert z\Vert+\Vert y\Vert)$$
holds for every $y\in Y$ and every $z\in Z$.
\end{lemma}

Observe that the mapping $T$ is a $(1+\varepsilon)$-isometry (just take $z=0$) and that $X$ is octahedral.

\begin{proof}
Since $Y$ is uniformly convex there exists a mapping $\delta:\mathbb R^+\longrightarrow \mathbb R^+$ such that $\lim_{\varepsilon\rightarrow 0}\delta(\varepsilon)=0$ and with the property that, given $\eta>0$, if $x,y\in B_Y$ satisfy $\Vert x+y\Vert>2-\delta(\eta)$ then $\Vert x-y\Vert<\eta$.

Take $\varepsilon>0$ and  $\eta>0$ small enough such that $\delta(\eta)+4\eta<\varepsilon$. Pick a finite dimensional subspace $Z$ of $X$. Take $\{y_1,\ldots, y_n\}$ a $\eta$-net of $S_Y$ and take $\{z_1,\ldots, z_p\}$ a $\eta$-net of $S_Z$. For every $i\in\{1,\ldots, n\}$ take $f_i\in S_{Y^*}$ such that $f_i(y_i)=1$.

Define $T_{ij}:=f_i\otimes z_j\in L(Y,X)$ by $T_{ij}(x):=f_i(x)z_j$, and note that $T_{ij}$ is a norm-one element. By the assumption that $L(Y,X)$ is octahedral we can find an operator $T\in S_{L(Y,X)}$ such that
$$\Vert T_{ij}+T\Vert>2-\delta(\eta)$$
holds for every $1\leq i\leq n$ and $1\leq j\leq p$. 

Fix $1\leq i\leq n$ and $1\leq j\leq p$. By the definition of the operator norm we can find $y_{ij}\in S_Y$ such that $2-\delta(\eta)<\Vert T_{ij}(y_{ij})+T(y_{ij})\Vert$. By the Hahn-Banach theorem we can find $x_{ij}^*\in S_{X^*}$ such that
$$2-\delta(\eta)<x_{ij}^*(f_i(y_{ij})z_j+T(y_{ij})).$$
Up to a change of sign we can assume with no loss of generality that $f_i(y_{ij})\geq 0$. Since all the elements in the above inequality are norm-one elements we get that $f_i(y_{ij})>1-\delta(\eta)$. Moreover, since $f_i(y_i)=1$ we get that $\Vert y_i+y_{ij}\Vert>2-\delta(\eta)$, and the uniform convexity implies that $\Vert y_i-y_{ij}\Vert<\eta$.
This implies that
\[\begin{split}2-\delta(\eta)<x_{ij}^*(f_i(y_{ij})z_j+T(y_{ij})) & \leq x_{ij}^*(f_i(y_{i})z_j+T(y_{i}))+2\Vert y_i-y_{ij}\Vert\\
& \leq \Vert z_j+T(y_i)\Vert +2\eta.
\end{split}\]
Since $\{y_1,\ldots, y_n\}$ is a $\eta$-net in $S_Y$ and $\{z_1,\ldots, z_p\}$ is a $\eta$-net in $S_Z$ we conclude that
$$\Vert z+T(y)\Vert>2-\delta(\eta)-4\eta>2-\varepsilon$$
holds for every $z\in S_Z$ and every $y\in S_Y$.

Let us conclude from here the desired result. To this end, take arbitrary $z\in S_Z$ and $y\in S_Y$, and take $t_1,t_2\in [0,1]$ such that $t_1+t_2=1$. Let us estimate $\Vert t_1 z+t_2 T(y)\Vert$. Assume with no loss of generality that $t_1\geq t_2$ (the other case runs similar). Then
\[\begin{split}
\Vert t_1z+t_2T(y)\Vert& =\Vert t_1(z+T(y))+(t_2-t_1)T(y)\Vert\geq t_1 \Vert z+T(y)\Vert-\vert t_2-t_1\vert\Vert T(y)\Vert\\
& \geq t_1(2-\varepsilon)+t_2-t_1=t_1+t_2-t_1\varepsilon\geq 1-\varepsilon.
\end{split}
\]
Observe that this proves in particular that $\Vert T(y)\Vert\geq 1-\varepsilon$ holds for every $y\in S_Y$ and, in consequence, $\Vert T(y)\Vert \geq (1-\varepsilon)\Vert y\Vert$ for every $y\neq 0$.

Now, given $z\in Z\setminus\{0\}$ and $y\in Y\setminus\{0\}$, we get that
$$\frac{\Vert z+T(y)\Vert}{\Vert z\Vert+\Vert T(y)\Vert }=\left\Vert\frac{\Vert z\Vert}{\Vert z\Vert+\Vert T(y)\Vert}\frac{z}{\Vert z\Vert}+\frac{\Vert T(y)\Vert}{\Vert z\Vert+\Vert T(y)\Vert}\frac{T(y)}{\Vert T(y)\Vert} \right\Vert>1-\varepsilon,$$
from where $\Vert z+ T(y)\Vert>(1-\varepsilon)(\Vert z\Vert+\Vert T(y)\Vert)>(1-\varepsilon)(\Vert z\Vert+(1-\varepsilon)\Vert y\Vert))>(1-\varepsilon)^2(\Vert z\Vert+\Vert y\Vert)$, and the lemma is proved.
\end{proof}

\begin{remark}\label{remark:homogeneity}
Observe that, from the last part of the above proof, the following holds true: Given two Banach spaces $X$ and $Y$ with $Y$ finite dimensional, the following assertions are equivalent:
\begin{enumerate}
\item For every finite dimensional subspace $Z$ of $X$ and every $\varepsilon>0$ we can find a norm-one operator $T:Y\longrightarrow X$ such that
$$\Vert z+T(y)\Vert\geq (1-\varepsilon)(\Vert z\Vert+\Vert y\Vert)$$
holds for every $y\in Y$ and every $z\in Z$.
\item For every finite subsets $\{z_1,\ldots, z_n\}\subseteq S_X$ and $\{y_1,\ldots, y_m\}\subseteq S_Y$ and every $\varepsilon>0$ there exists a norm-one operator $T:Y\longrightarrow X$ such that
$$\Vert z_i+T(y_j)\Vert>2-\varepsilon$$
holds for every $1\leq i\leq n$ and $1\leq j\leq m$.
\end{enumerate}
We will use this remark throughout the text.
\end{remark}

Now we are ready to prove the following necessary condition for a Banach space being universally octahedral.

\begin{theorem}\label{theo:condinece}
Let $X$ be a universally octahedral Banach space. Then, for every $\varepsilon>0$, for every finite dimensional subspace $Z$ of $X$ and for every $n\in\mathbb N$, there exists an operator $T:\ell_\infty^n\longrightarrow X$ such that $\Vert T\Vert\leq 1$ and such that
$$\Vert z+T(y)\Vert\geq (1-\varepsilon)(\Vert z\Vert+\Vert y\Vert)$$
holds for every $y\in \ell_\infty^n$ and every $z\in Z$.
\end{theorem}

\begin{proof}
Observe that, given $x\in \mathbb R^n$, we have that $\Vert x\Vert_\infty\leq \Vert x\Vert_p\leq n^\frac{1}{p}\Vert x\Vert_\infty$, so $\ell_\infty^n$ is, for every $\varepsilon>0$, $(1+\varepsilon)$-isometric to a uniformly convex Banach space.

The proof is simple from now. Take $\varepsilon>0$ and $n\in\mathbb N$, and take $p\in\mathbb N$ such that a suitable scalling of the formal identity $\phi:\ell_\infty^n\longrightarrow \ell_p^n$ satisfies 
$$\sqrt{1-\varepsilon}\Vert x\Vert\leq \Vert \phi(x)\Vert\leq \Vert x\Vert\ \forall x\in X.$$
Now set $Z\subseteq X$ be a finite dimensional subspace. Applying Lemma \ref{lemma:uconvexo} we can find a bounded operator $T:\ell_p^n\longrightarrow X$ with $\Vert T\Vert\leq 1$ and such that
$$\Vert z+T(y)\Vert\geq \sqrt{1-\varepsilon}(\Vert z\Vert+\Vert y\Vert)$$
holds for every $z\in Z$ and every $y\in Y$. Now $T\circ\phi$ is the desired operator.
\end{proof}

\begin{remark}\label{remark:condinece} A couple of remarks are pertinent.
\begin{enumerate}
\item In the above proof we have only used that $L(\ell_p^n,X)$ is octahedral for every $n\in\mathbb N$ and every $1<p<\infty$.

\item Observe that this in particular implies that $c_0$ is finitely representable in $X$ \cite[Lemma 11.1.6]{alka}, which implies that every Banach space is finitely representable in $X$. In particular $\ell_\infty$ is finitely representable in $X$, which implies that $X$ has a trivial cotype \cite[Theorem 11.1.14]{alka}.
\end{enumerate}
\end{remark}

Now it is time to prove that the converse holds true.

\begin{theorem}\label{theo:condisufi}
Let $X$ be a Banach space. Assume that, for every $\varepsilon>0$, for every finite dimensional subspace $Z$ of $X$ and for every $n\in\mathbb N$, there exists an operator $T:\ell_\infty^n\longrightarrow X$ such that $\Vert T\Vert\leq 1$ and such that
$$\Vert z+T(y)\Vert\geq (1-\varepsilon)(\Vert z\Vert+\Vert y\Vert)$$
holds for every $y\in \ell_\infty^n$ and every $z\in Z$.

Then, for every Banach space $Y$ and for every subspace $H$ of $L(Y,X)$ containing the finite rank operators, the norm of $H$ is octahedral.
\end{theorem}

\begin{proof}
Let $Y$ be a non-zero Banach space and $H\subseteq L(Y,X)$ as in the hypothesis. In order to prove that $H$ is octahedral pick $T_1,\ldots, T_n\in S_H$ and $\varepsilon>0$, and let us find an element $\Psi\in S_H$ such that $\Vert T_i+\Psi\Vert>2-\varepsilon$ holds for every $1\leq i\leq n$. This is enough by \cite[Proposition 2.1]{hlp}. In order to do so find, for every $i\in\{1,\ldots, n\}$, an element $y_i\in S_Y$ such that $\Vert T_i(y_i)\Vert>1-\varepsilon$. Set $Z:=\spann\{T(y_i): 1\leq i\leq n\}$.

Set also $V:=\spann\{y_1,\ldots, y_n\}\subseteq Y$. Since every Banach space is finitely representable in $c_0$ we can find an operator $\phi:V\longrightarrow c_0$ with $\Vert \phi(y_i)\Vert>1-\varepsilon$ for every $i$ and such that $\Vert \phi\Vert<1$. This operator can be extended by Theorem \ref{theo:lindenstrauss} to an operator $Q:Y\longrightarrow c_0$ which satisfies that $\Vert Q(y_i)\Vert>1-\varepsilon$ for every $i$ and still $\Vert Q\Vert<1$. By the definition of the $c_0$ norm we can find $n$ large enough such that, if we define $P:c_0\longrightarrow \ell_\infty^n$ the natural projection, we get $\Vert P(Q(y_i))\Vert>1-\varepsilon$ for every $i$. 

Now, by the hypothesis, we can find an operator $T:\ell_\infty^n\longrightarrow X$ such that $\Vert T\Vert\leq 1$ and such that
$$\Vert z+T(y)\Vert\geq (1-\varepsilon)(\Vert z\Vert+\Vert y\Vert)$$
holds for every $y\in \ell_\infty^n$ and every $z\in Z$.

Now the desired operator is $\Psi:=T\circ P\circ Q:Y\longrightarrow X$, which belongs to $F(Y,X)\subseteq H$. Observe that $\Vert \Psi\Vert\leq 1$. Moreover, given $1\leq i\leq n$, we get
\[
\begin{split}
\Vert T_i+\Psi\Vert\geq \Vert T_i(y_i)+T(Q(P(y_i)))\Vert & \geq (1-\varepsilon)(\Vert T_i(y_i)\Vert+\Vert Q(P(y_i))\Vert)\\
& \geq (1-\varepsilon)(1-\varepsilon+1-\varepsilon)=2(1-\varepsilon)^2.
\end{split}
\]
Since $\varepsilon$ was arbitrary we conclude the result.
\end{proof}

As a consequence we get the following result.

\begin{theorem}\label{theo:carauniocta}
Let $X$ be a Banach space. The following are equivalent:
\begin{enumerate}
\item For every Banach space $Y$ and every $H\subseteq L(Y,X)$ containing the finite-rank operators, the space $H$ is octahedral.
\item $X$ is universally octahedral.
\item For every finite dimensional Banach space $Y$, the space $L(Y,X)$ is octahedral.
\item For every finite dimensional uniformly convex Banach space $Y$, the space $L(Y,X)$ is octahedral.
\item For every $1<p<\infty$ and every $n\in\mathbb N$ the space $L(\ell_p^n,X)$ is octahedral.
\item For every $\varepsilon>0$, for every finite dimensional subspace $Z$ of $X$ and for every $n\in\mathbb N$, there exists an element $T:\ell_\infty^n\longrightarrow X$ with $\Vert T\Vert\leq 1$ and such that
$$\Vert z+T(y)\Vert\geq (1-\varepsilon)(\Vert z\Vert+\Vert y\Vert)$$
holds for every $y\in \ell_\infty^n$ and every $z\in Z$.
\end{enumerate}
\end{theorem}

\begin{proof}
(1)$\Rightarrow$(2)$\Rightarrow$(3)$\Rightarrow$(4)$\Rightarrow$(5) are immediate, and (5)$\Rightarrow$(6) follows by Remark \ref{remark:condinece}. Finally, (6)$\Rightarrow$(1) is Theorem \ref{theo:condisufi}.
\end{proof}

Observe that condition (6) requires not only that $\ell_\infty$ is finitely representable in $X$, but also that $\ell_\infty$ must be finitely representable in a part of $X$ which kind of $\ell_1$-orthogonal to $Z$ for any finite-dimensional subspace $Z$ of $X$. This fact will become more clear in the following example.

\begin{example}\label{exam:nounivoctalinf}
Let $X:=\ell_\infty\oplus_1 \ell_1$. $X$ is octahedral \cite[Proposition 3.10]{hlp} and clearly contains $\ell_\infty$ isometrically. However, we claim that $X$ is not universally octahedral.

Indeed, assume by contradiction that $X$ is universally octahedral. Let $Z:=\spann\{(e_1,0)\}\subseteq X=\ell_\infty\oplus_1 \ell_1$. 
Fix $n\in\mathbb N$ and $\varepsilon>0$. By the above characterisation we can find $\phi:\ell_\infty^n\longrightarrow X$ with $\Vert \phi\Vert\leq 1$ and such that
$$\Vert (e_1,0)\pm\phi(y)\Vert>(1-\varepsilon)(1+\Vert y\Vert)$$
holds for every $y\in \ell_\infty^n$. 

Take $Q:X=\ell_\infty\oplus_1 \ell_1\longrightarrow \ell_1$ the natural projection. Let us prove that $Q\circ\phi: \ell_\infty^n\longrightarrow \ell_1$ satisfies that
$$(1-2\varepsilon)\Vert y\Vert\leq \Vert Q(\phi(y))\Vert\leq \Vert y\Vert,$$
from where the arbitrariness of $n$ and $\varepsilon$ will imply that $\ell_\infty$ is finitely-representable in $\ell_1$, which is a contradiction because $\ell_1$ has cotype 2 and the finite representability of $\ell_\infty$ in a Banach space $Z$ implies that $Z$ fails to have cotype $q$ for every $q<\infty$ \cite[Theorem 11.1.14]{alka}.

So take $x\in S_{\ell_\infty^n}$, call $\phi(x):=(a,b)\in \ell_\infty\oplus_1 \ell_1$ and notice that $\Vert a\Vert_\infty+\Vert b\Vert_1\leq 1$. Note that
$$2(1-\varepsilon)<\Vert (e_1,0)\pm \phi(x)\Vert=\Vert (e_1\pm a)\Vert_\infty+\Vert b\Vert_1.$$
It is direct computation that either $\Vert e_1+a\Vert_\infty\leq 1$ or $\Vert e_1-a\Vert_\infty\leq 1$. Assume without loss of generality that $\Vert e_1+a\Vert_\infty\leq 1$. Now
$$2(1-\varepsilon)\leq \Vert e_1+a\Vert_\infty+\Vert b\Vert_1\leq 1+\Vert b\Vert_1,$$
which implies $\Vert b\Vert_1\geq 1-2\varepsilon$.

This implies that $\Vert Q(\phi(x))\Vert\geq 1-2\varepsilon$. The arbitrariness of $x\in S_{\ell_\infty^n}$ forces that $\Vert Q(\phi(x))\Vert\geq 1-2\varepsilon$ holds for every $x\in S_{\ell_\infty^n}$, and a homogeneity argument yields that, for every $x\in \ell_\infty^n$, we get
$$(1-2\varepsilon)\Vert x\Vert\leq \Vert Q(\phi(x))\Vert\leq \Vert Q\Vert \Vert \phi\Vert \Vert x\Vert\leq \Vert x\Vert,$$
as desired.
\end{example}

Another exotic example is the following.

\begin{example}\label{example:uninoc0octa}
Let $X:=(\oplus_{i=1}^\infty \ell_\infty^i)_1$. It is immediate by the main characterisation that $X$ is universally octahedral. Indeed, given $n\in\mathbb N$, $\varepsilon>0$ and $\{z_1,\ldots, z_k\}\subseteq S_X$, it is enough by Remark \ref{remark:homogeneity} to find an operator $T:\ell_\infty^n\longrightarrow X$ with
$$\Vert z_i+T(y)\Vert>(1-\varepsilon)(1+\Vert y\Vert)$$
for every $1\leq i\leq k$ and every $y\in S_{\ell_\infty^n}$. Up to a density argument we can assume that $z_i\in (\oplus_{i=1}^\infty \ell_\infty^i)$ have finite support (say contained in $\oplus_{i=1}^p \ell_\infty^i$). Take $q>\max\{p, n\}$. Now take the canonical inclusion operator $\phi:\ell_\infty^n\hookrightarrow \ell_\infty^q$. Take the canonical inclusion operator $j:\ell_\infty^q\hookrightarrow \oplus_{i=1}^\infty \ell_\infty^i$ by $j(x)(q)=x$ if $i=q$ and $0$ otherwise. Now $T=j\circ\phi: \ell_\infty^n\longrightarrow X$ satisfies the desired requirement. In fact, given any $y\in \ell_\infty^n$ observe that, since the support of $z_i$ and $T(y)$ are disjoint by construction, we obtain that
$$\Vert z_i+T(y)\Vert=\Vert z_i\Vert+\Vert T(y)\Vert=1+\Vert y\Vert$$
since $T$ is an isometry. This proves that $X$ is universally octahedral.

However $X$ does not contain $c_0$ isomorphically. Indeed, $X=(\oplus_{i=1}^\infty \ell_1^i)_\infty^*$ is a dual space, so if $X$ contained $c_0$ then it would indeed contain $\ell_\infty$ by \cite[Theorem 6.39]{checos}, which is impossible since $X$ is clearly separable.
\end{example}

Let us end with an observation concerning the existence of $L$-orthogonal elements.

\begin{remark}\label{remark:Lorthogonaluniversal} Let $X$ be a Banach space. Following the notation of \cite{loru2021}, we say that an element $u\in X^{**}$ is an \textit{$L$-orthogonal element} if it satisfies
$$\Vert x+u\Vert=\Vert x\Vert+\Vert u\Vert$$
holds for every $x\in X$.

The existence of non-zero $L$-orthogonal elements is strongly connected with octahedral norms. It is a consequence of the Principle of Local Reflexivity (and explicitly mentioned in  \cite[Lemma 9.1]{gk}) that if $X$ has a non-zero $L$-orthogonal element then the norm of $X$ is octahedral. Moreover, the converse is true if $X$ is separable \cite[Lemma 9.1]{gk}. The question whether octahedrality implies the existence of non-zero $L$-orthogonals has remained open until the recent work \cite{loru2021}, where many examples of octahedral spaces without any non-zero $L$-orthogonal element is exhibited. 

A natural question at this point is whether or not there exists a Banach space $X$ satisfying that, for every non-zero Banach space $Y$ and for every $H\subseteq L(Y,X)$ containing $F(Y,X)$, the space $H$ has non-zero $L$-orthogonal elements.

The answer is no. Indeed, given any Banach space $X$, taking $Y=\ell_2(I)$ and $H$ to be the space of compact operators from $Y$ to $X$ (denoted by $K(Y,X)$), if $H$ has a non-zero $L$-orthogonal element, then $Y$ is isometrically isomorphic to a subspace of $X^{**}$ by \cite[Lemma 3.1]{loru2021}, so it remains to take $I$ big enough so that there is no injective mapping $\phi:I\longrightarrow X^{**}$ to conclude that $K(\ell_2(I),X)$ does not have any non-zero $L$-orthogonal element.
\end{remark}

\section{Examples}\label{section:examples}

In this section we will analyse examples of universally octahedral Banach spaces. Observe that the condition in Theorem \ref{theo:condisufi} is very difficult to check in a particular example. So, in order to provide examples where universal octahedrality holds, we give a criterion which implies it. In order to save notation, let us make the following definition.

\begin{definition}\label{defi:c0octa}
Let $X$ be a Banach space. We say that $X$ is \textit{$c_0$-octahedral} if, for every $x_1,\ldots, x_n\in S_X$ and every $\varepsilon>0$, there exists a subspace $Y\subseteq X$ such that
\begin{enumerate}
\item $\Vert x_i+y\Vert>(1-\varepsilon)(1+\Vert y\Vert)$ holds for every $y\in Y$ and;
\item $Y$ is isometric to $c_0$.
\end{enumerate}
\end{definition}

\begin{remark}\label{rem:c0octaimpliesunivers}
It is quite easy to prove (up to applying Remark \ref{remark:homogeneity}) that $c_0$-octahedrality implies universal octahedrality. Moreover, a homogeneity argument similar to that of the end of Lemma \ref{lemma:uconvexo} allows to show that a Banach space $X$ is $c_0$-octahedral if, and only if, for every finite dimensional subspace $Z\subseteq X$ and every $\varepsilon>0$, there exists a subspace $Y\subseteq X$ such that
\begin{enumerate}
\item $\Vert z+y\Vert>(1-\varepsilon)(\Vert z\Vert+\Vert y\Vert)$ holds for every $y\in Y$ and $z\in Z$ and;
\item $Y$ is isometric to $c_0$.
\end{enumerate}
\end{remark}

\begin{example}\label{example:Lip}
By \cite[Proposition 3.3]{laru2020} we have that if $\Lip(M)$ is octahedral then $\Lip(M)$ is $c_0$-octahedral (and consequently universally octahedral).
\end{example}

In general, there are Banach spaces which are universally octahedral but not $c_0$-octahedral, and one example is given in Example \ref{example:uninoc0octa}. Consequently, if $X$ is universally octahedral it does not imply that $X$ is $c_0$-octahedral. However, we can guarantee that all its ultrapowers are $c_0$-octahedral.

\begin{proposition}\label{prop:ultrapowers}
Let $X$ be a Banach space and let $\mathcal U$ be a free ultrafilter over $\mathbb N$. Then $X$ is universally octahedral if, and only if, $X_\mathcal U$ is $c_0$-octahedral.
\end{proposition}

\begin{proof}
Let us start assuming that $X$ is universally octahedral and let us prove that $X_\mathcal U$ is $c_0$-octahedral. To prove $c_0$-octahedrality, pick $(z_n^1),\ldots, (z_n^p)\in S_{X_\mathcal U}$. We can assume, up to changing of representative, that $\Vert z_n^i\Vert=1$ holds for every $1\leq i\leq p$ and $n\in\mathbb N$. Since $X$ is universally octahedral then, for every $n\in\mathbb N$ we can find an operator $T_n:\ell_\infty^n\longrightarrow X$ with $\Vert T_n\Vert\leq 1$ such that
$$\Vert z_n^i+T_n(x)\Vert\geq \left(1-\frac{1}{n}\right)(1+\Vert x\Vert)$$
holds for every $x\in \ell_\infty^n$. Now define $T:(\ell_\infty^n)_\mathcal U\longrightarrow X_\mathcal U$ by the equation
$$T((x_n)):=(T_n(x_n))\ \forall (x_n)\in (\ell_\infty^n)_\mathcal U.$$
It is clear that $\Vert T\Vert\leq 1$. Moreover, given $1\leq i\leq p$ and $(x_n)\in (\ell_\infty^n)_\mathcal U$, we get
\[\begin{split}
\Vert (z_n^i)+T((x_n))\Vert_\mathcal U=\lim_\mathcal U \Vert z_n^i+T_n(x_n)\Vert& \geq \lim_\mathcal U \left(1-\frac{1}{n}\right)(1+\Vert x_n\Vert)\\
& =1+\lim_\mathcal U \Vert x_n\Vert=1+\Vert (x_n)\Vert_\mathcal U,
\end{split}\]
from where $\Vert (z_n^i)+T((x_n))\Vert_\mathcal U=1+\Vert (x_n)\Vert_\mathcal U$. In particular, $T$ is an isometry so, in order to finish the proof, it remains to show that $(\ell_\infty^n)_\mathcal U$ contains an isometric copy of $c_0$. Take $k\in\mathbb N$ and define $(x_n^k)\in (\ell_\infty^n)_\mathcal U$ by
$$x_n^k:=\left\{\begin{array}{cc}
e_k & \mbox{ if }n\geq k,\\
0 &\mbox{ otherwise.}
\end{array} \right.$$
It is not difficult to prove that, given $\lambda_1,\ldots, \lambda_k\in\mathbb R$, for $n\geq k$ we obtain that $\Vert \sum_{i=1}^k \lambda_i x_n^i\Vert=\max\limits_{1\leq i\leq k} \vert \lambda_i\vert$, so
$$\left\Vert \sum_{i=1}^k\lambda_i (x_n^i)\right\Vert_\mathcal U=\lim_\mathcal U \left\Vert \sum_{i=1}^k \lambda_i x_n^i\right\Vert=\max\limits_{1\leq i\leq k} \vert \lambda_i\vert,$$
which proves that $(x_n^k)$ is isometric to the $c_0$-basis. It remains to take $Y:=\overline{\spann}\{T(x_n^k):k\in\mathbb N\}$ to get the desired subspace.

For the converse, let $z_1,\ldots, z_q\in S_X$, $\varepsilon>0$ and $n\in\mathbb N$, and let us find $T:\ell_\infty^n\longrightarrow X$ with $\Vert T\Vert\leq 1+\varepsilon$ and such that
$$\Vert z_i+T(y)\Vert\geq (1-\varepsilon)(1+\Vert y\Vert)$$
holds for every $1\leq i\leq q$. Call $j:X\longrightarrow X_\mathcal U$ the natural embedding by $j(x):=(x)_\mathcal U$. Take $\nu\in\mathbb R^+$ small enough such that $(1+\nu)(1-3\nu)^{-1}<1+\varepsilon$ and such that $\nu (1+\varepsilon)<\varepsilon$. Since $X_\mathcal U$ is $c_0$-octahedral we can find $Y\subseteq X_\mathcal U$ such that $Y$ is isometric to $c_0$ and such that
$$\Vert j(z_i)+x_k\Vert_\mathcal U>\left(1-\frac{\nu}{2}\right)(1+\Vert (x_k)\Vert_\mathcal U)=\left(1-\frac{\nu}{2}\right)(1+\lim_\mathcal U \Vert x_k\Vert)$$
holds for every $(x_n)\in Y$. Since $Y$ is isometric to $c_0$ there exists an isometry $\phi:\ell_\infty^n\longrightarrow Y$. Take $F\subseteq B_{\ell_\infty^n}$ a $\nu$-net and observe that
$$\lim_\mathcal U \Vert z_i+\phi(z)(k)\Vert\geq \left(1-\frac{\nu}{2}\right)(1+\Vert z\Vert)$$
holds for every $z\in F$. Since $F$ is finite we can find $A\in\mathcal U$ such that 
$$(1-\nu)(1+\Vert z\Vert)\leq \Vert z_i+\phi(z)(k)\Vert\leq (1+\nu)(1+\Vert z\Vert)$$
and
$$(1-\nu)\Vert z\Vert\leq \Vert \phi(z)(k)\Vert\leq (1+\nu)\Vert z\Vert$$
holds for every $z\in F$ and every $k\in A$. 

Now select any $k\in A$ and define a linear operator $T:\ell_\infty^n\longrightarrow X$ by $T(x):=\phi(x)(k)$ for every $x\in \ell_\infty^n$.

Observe, on the one hand, that given $z\in F$, we get
$$(1-\nu)\Vert z\Vert\leq \Vert T(z)\Vert\leq (1+\nu)\Vert z\Vert$$
and, since $F$ is a $\nu$-net of $B_{\ell_\infty^n}$ and $(1+\nu)(1-3\nu)^{-1}<1+\varepsilon$, we get from Lemma \ref{lema:caraepsisometrynets} that
$$(1-\varepsilon)\Vert z\Vert\leq \Vert T(z)\Vert\leq (1+\varepsilon)\Vert z\Vert $$
holds for every $z\in S_{\ell_\infty^n}$.
Now, given $x\in B_{\ell_\infty^n}$, take $z\in F$ such that $\Vert x-z\Vert<\nu$.
\[
\begin{split}
\Vert z_i+T(x)\Vert\geq \Vert z_i+T(z)\Vert-\Vert T(x-z)\Vert& >(1-\nu)(1+\Vert z\Vert)-(1+\varepsilon)\Vert z-x\Vert\\
& \geq (1-\nu)(1+\Vert x\Vert-\Vert x-z\Vert)-(1+\varepsilon)\nu\\
& > (1-\nu)(1+\Vert x\Vert-\nu)-(1+\varepsilon)\nu\\
& \geq (1-\nu)(1+\Vert x\Vert -(1+\Vert x\Vert)\nu)-\varepsilon\\
& \geq (1-\nu)^2(1+\Vert x\Vert)-(1+\Vert x\Vert)\varepsilon\\
& = ((1-\nu)^2-\varepsilon)(1+\Vert x\Vert).
\end{split}
\]
Now given $x\in \ell_\infty^n$ with $\Vert x\Vert>1$ we finally get
\[
\begin{split}
\Vert z_i+T(x)\Vert& =\left\Vert \Vert x\Vert \left(z_i+T\left(\frac{x}{\Vert x\Vert}\right)\right)-(\Vert x\Vert-1)z_i\right\Vert\\
& \geq \Vert x\Vert \left\Vert z_i+T\left(\frac{x}{\Vert x\Vert}\right) \right\Vert-(\Vert x\Vert-1)\Vert z_i\Vert\\
& >\Vert x\Vert 2((1-\nu)^2-\varepsilon)-\Vert x\Vert+1\\
& =\Vert x\Vert (2-4\nu+2\nu^2-2\varepsilon)-\Vert x\Vert+1\\
& =\Vert x\Vert (1-4\nu+2\nu^2-2\varepsilon)+1\\
& > (1-4\nu+2\nu^2-2\varepsilon)(\Vert x\Vert+1),
\end{split}
\]
and the proof is finished.
\end{proof}

\begin{remark}
Observe that in the above proof we have obtained in $X_\mathcal U$ $c_0$-octahedrality with the constant $\varepsilon=0$. This is not surprising because J. D. Hardtke proved \cite[Proposition 2.10]{hardtke} that, given any free ultrafilter $\mathcal U$ over $\mathbb N$,  a Banach space $X$ is octahedral if, and only if, $X_\mathcal U$ satisfies that for every $(z_n^1),\ldots, (z_n^k)\in S_{X_\mathcal U}$ there exists $(z_n)\in S_{X_\mathcal U}$ such that $\Vert (z_n^k)+(z_n)\Vert=2$.
\end{remark}

We end the paper by proving that $L_1$-preduals with octahedral norms are actually universally octahedral, which enlarges our class of Banach spaces which are universally octahedral.

\begin{theorem}\label{theo:L1preduals}
Let $X$ be an $L_1$-predual. If $X$ is octahedral then $X$ is universally octahedral.
\end{theorem}

For the proof we will need the following lemma which says that octahedral $L_1$-preduals are close to be $c_0$-octahedral when we look at the bidual space.

\begin{lemma}\label{lemma:L1preduoctae}
Let $X$ be an $L_1$ predual. If $X$ is octahedral then, for every $x_1,\ldots, x_k\in S_X$ and every $\varepsilon>0$ there exists a sequence $\{x_n^{**}\}\subseteq S_{X^{**}}$ such that
\begin{enumerate}
\item $\{x_n^{**}\}$ is isometric to the usual basis of $c_0$ and,
\item $\Vert x_i+v\Vert>1-\varepsilon+\Vert v\Vert$ holds for every $1\leq i\leq k$ and $v\in \overline{\spann}\{x_n^{**}: n\in\mathbb N\}$.
\end{enumerate}
\end{lemma}

\begin{proof}
Since $X$ is octahedral then clearly $X$ does not have any point of Fr\'echet differentiability. Since $X$ is an $L_1$-predual we get that $X$ has the Daugavet property \cite[Theorem 2.4]{becemartin}. 

By \cite[Theorem 3.2]{mr2022} we conclude that, for every $\varepsilon>0$, the set
$$V_i^\pm:=\{e^*\in \ext{B_{X^*}}: e^*(\pm x_i)>1-\varepsilon\}$$
is infinite for every $i\in\{1,\ldots, k\}$, where $\ext{B_{X^*}}$ stands for the set of extreme points of $B_{X^*}$.

Consequently we can construct, by an inductive argument, sequences $\{e_{in}^{\pm*}\}_{n\in\mathbb N}$ of elements such that $e_{in}^{\pm*}\in V_i^\pm$ for every $n\in\mathbb N$ and $1\leq i\leq k$ and such that $e_{in}^{\pm*}$ and $e_{jm}^{\pm*}$ are linearly independent if $(i,n)\neq (j,m)$ and, moreover, $e_{ij}^{+*}$ and $e_{ij}^{-*}$ are linearly independent too for every $i,j$.

Since $X^*=L_1(\mu)$ for certain $\mu$ it follows that, for every $1\leq i\leq k$ and every $n\in\mathbb N$, there exists an $L$-projection $P_{in}^\pm:X^*\longrightarrow X^*$ such that $X^*:=\mathbb R e_{in}^{\pm*}\oplus_1 \ker(P_{in}^\pm)$. Now set, for every $n\in\mathbb N$, $P_n:=\sum_{i=1}^k (P_{in}^++P_{in}^-)$. By \cite[Lemma 2.2]{mr2022} we get that $P_n:X^*\longrightarrow X^*$ is an $L$-projection with $B_{P_n(X^*)}=\aconv\{e_{1n}^{\pm*},\ldots, e_{kn}^{\pm*}\}$ (consequently $P_n(X^*)=\spann\{e_{1n}^{\pm*},\ldots, e_{kn}^{\pm*}\}$ is isometrically $\ell_1^{2k}$) and such that $\ker(P_n)=\bigcap\limits_{i=1}^k \ker(P_{in}^+)\cap \ker(P_{in}^-)$. Consequently $X^*:=P_n(X^*)\oplus_1 \ker(P_n)=\spann\{e_{1n}^{\pm*},\ldots, e_{kn}^{\pm*}\}\oplus_1 \ker(P_n)$.

Given $n\in\mathbb N$ define $x_n^{**}:X^*\longrightarrow \mathbb R$ by the equation
$$x_n^{**}\left(\sum_{i=1}^k (\lambda_i^+ e_{in}^{+*}+\lambda_i^-e_{in}^{-*})+z\right)=\sum_{i=1}^k \lambda_i^++\lambda_i^-\ $$
for all $\sum_{i=1}^k (\lambda_i^+ e_{in}^{+*}+\lambda_i^-e_{in}^{-*})+z\in \spann\{e_{1n}^{\pm*},\ldots, e_{kn}^{\pm*}\}\oplus_1 \ker(P_n)=X^*.$
We claim that $\Vert x_n^{**}\Vert\leq 1$. Indeed, given $\sum_{i=1}^k (\lambda_i^+ e_{in}^{+*}+\lambda_i^-e_{in}^{-*})+z\in \spann\{e_{1n}^{\pm*},\ldots, e_{kn}^{\pm*}\}\oplus_1 \ker(P_n)=X^*$, we have that
\[
\begin{split}
\left\vert x_n^{**}\left( \sum_{i=1}^k (\lambda_i^+ e_{in}^{+*}+\lambda_i^-e_{in}^{-*})+z\right) \right\vert& =\left\vert \sum_{i=1}^k \lambda_i^++\lambda_i^-\right\vert \leq \sum_{i=1}^k \vert \lambda_i^+\vert+\vert \lambda_i^-\vert  \\
& \leq  \left\Vert \sum_{i=1}^k \lambda_i^+ e_{in}^{+*}+\lambda_i^- e_{in}^{-*}\right\Vert\\
& \leq \left\Vert \sum_{i=1}^k \lambda_i^+ e_{in}^{+*}+\lambda_i^- e_{in}^{-*}\right\Vert+\Vert z\Vert\\
& =\left\Vert \sum_{i=1}^k \lambda_i^+ e_{in}^{+*}+\lambda_i^- e_{in}^{-*}+z\right\Vert,
\end{split}
\]
which proves that $\Vert x_n^{**}\Vert\leq 1$. Let us prove that $\{x_n^{**}\}$ has the desired properties.

Now it is time to prove that $\{x_n^{**}\}$ is isometric to the $c_0$ basis. To this end pick $\lambda_1,\ldots \lambda_n\in\mathbb R$ and let us estimate the norm of $\sum_{i=1}^n \lambda_i x_i^{**}$. On the one hand, given $1\leq i\leq n$ such that $\max_{1\leq j\leq n}\vert \lambda_j\vert=\vert \lambda_i\vert$, set $\sigma_i:=\operatorname{sign}(\lambda_i)$ (with the convention $\operatorname{sign}(0)=1$). Moreover, given $1\leq i\leq n$, define $\tau_i:=+$ if $\sigma_i=1$ and $\tau_i:=-$ if $\sigma(i)=-1$. Then
$$\left\Vert \sum_{i=1}^n \lambda_ix_i^{**}\right\Vert\geq \sum_{i=1}^n \lambda_i x_i^{**}(\sigma_i e_{in}^{\tau_i *})=\lambda_i \sigma_i=\vert \lambda_i\vert,$$
because clearly $e_{in}\in \bigcap_{j\neq i} \ker(P_j)$.

For the inequality from above, we can apply again \cite[Lemma 2.2]{mr2022} to derive that $P:=\sum_{j=1}^n P_j=\sum_{j=1}^n \sum_{i=1}^k (P_{ij}^++P_{ij}^-)$ is an $L$-projection, that $B_{P(X^*)}=\aconv\{e_{ij}^{\pm *}: 1\leq i\leq k, 1\leq j\leq n \}$ and such that $\ker(P)=\bigcap\limits_{j=1}^n\bigcap_{i=1}^k P_{ij}^\pm$. Now
$$\left\Vert \sum_{j=1}^n \lambda_j x_j^{**}\right\Vert=\sup_{x^*\in B_{X^*}} \sum_{j=1}^n \lambda_j x_j^{**}(x^*).$$
The decomposition $X^*=P(X^*)\oplus_1 \ker(P)$ implies that $B_{X^*}=\overline{\co}(B_{P(X^*)}\cup B_{\ker(P)})=\{\lambda u+(1-\lambda)v: u\in B_{P(X^*)}, v\in B_{\ker(P)}, \lambda\in [0,1]\}$. Since by construction $\sum_{j=1}^n \lambda_j x_j^{**}$  vanishes on $\ker(P)$ we get that
$$\sup_{x^*\in B_{X^*}} \sum_{j=1}^n \lambda_j x_j^{**}(x^*)=\sup_{x^*\in B_{P(X^*)}} \sum_{j=1}^n \lambda_j x_j^{**}(x^*)$$
Moreover, since $B_{P(X^*)}=\aconv(\{e_{ij}^{\pm*}: 1\leq i\leq k, 1\leq j\leq n\})$ we get that
$$\sup_{x^*\in B_{P(X^*)}} \sum_{j=1}^n \lambda_j x_j^{**}(x^*)=\max\limits_{1\leq i\leq k, 1\leq p\leq n}\left\vert\sum_{j=1}^n\lambda_j x_j^{**}(e_{ip}^{\pm*})  \right\vert,$$
and since $x_j^{**}(e_{ip}^{\pm*})=\delta_{jp}$ we concluce that the above maximum coincides with $\max\limits_{1\leq j\leq n}\vert\lambda_j\vert$, as desired.

In order to finish the proof, take $v\in \spann\{x_i^{**}\}$, so $v=\sum_{i=1}^n \lambda_i x_i^{**}$ with $\Vert v\Vert=\max\limits_{1\leq i\leq n} \vert \lambda_i\vert=\vert \lambda_j\vert$ for some $j$. Let us prove that $\Vert x_i+v\Vert>1-\varepsilon+\vert \lambda_j\vert$ and the proof will be finished. To this end, let $1\leq p\leq k$. If $\lambda_j\geq 0$ then
$$\left\Vert x_p+\sum_{i=1}^n \lambda_i x_i^{**}\right\Vert \geq x_p(e_{pj}^{+*})+\sum_{i=1}^n \lambda_i x_i^{**}(e_{pj}^{+*})>1-\varepsilon+\lambda_j=1-\varepsilon+\vert \lambda_j\vert.$$
On the other hand if $\lambda_j<0$ we get
$$x_p(e_{pj}^{-*})+\sum_{i=1}^n \lambda_i x_i^{**}(e_{pj}^{-*})<-1+\varepsilon+\lambda_j=-1+\varepsilon-\vert \lambda_j\vert=-(1-\varepsilon+\vert \lambda_j\vert),$$
from where $\Vert x_p+\sum_{i=1}^n \lambda_i x_i^{**}\Vert\geq 1-\varepsilon+\vert\lambda_j\vert$ too, and the proof is finished.
\end{proof}

\begin{proof}[Proof of Theorem \ref{theo:L1preduals}]
Let $x_1,\ldots, x_k\in S_X$ and $\varepsilon>0$. By Lemma \ref{lemma:L1preduoctae} there exists a sequence $\{x_n^{**}\}\subseteq S_{X^{**}}$ such that $\Vert x_i+v\Vert>1-\varepsilon+\Vert v\Vert$ holds for every $1\leq i\leq k$ and any $v\in \spann\{x_n^{**}: n\in\mathbb N\}$, and such that $\{x_n^{**}\}$ is isometric to the $c_0$ basis.

Now let $n\in\mathbb N$ and take $\phi:\ell_\infty^n\longrightarrow \spann\{x_j^{**}:j\in\mathbb N\}$ the canonical inclusion $\phi(e_j):=x_j^{**}$ for $1\leq j\leq n$. 

Now set $E:=\spann\{x_1,\ldots, x_k, x_1^{**},\ldots, x_n^{**}\}\subseteq X^{**}$ and $\varepsilon>0$. By the Principle of Local Reflexivity we can find an operator $P: E\longrightarrow X$ such that $P(e)=e$ for $e\in E\cap X$ and 
$$(1-\varepsilon)\Vert x\Vert\leq \Vert P(x)\Vert\leq \Vert x\Vert$$
holds for every $x\in E$. Now set $T:=P\circ\phi:\ell_\infty^n\longrightarrow X$. Given $y\in \ell_\infty^n$ we have
\[
\begin{split}
\Vert x_i+T(y)\Vert=\Vert P(x_i)+P(\phi(y))\Vert=\Vert P(x_i+\phi(y))\Vert&\geq (1-\varepsilon)(\Vert x_i+\phi(y)\Vert)\\
& \geq (1-\varepsilon)(1-\varepsilon+\Vert \phi(y)\Vert)\\
& =(1-\varepsilon)^2(1 +\Vert y\Vert)
\end{split}
\]
The arbitrariness of $\varepsilon$ and $y$ yields the desired conclusion.
\end{proof}

Observe that we have proved that, given an $L_1$-predual $X$, we have that $X$ is universally octahedral if, and only if, $X$ is octahedral and if, and only if, $X$ has the Daugavet property. See \cite{mr2022} for examples of $L_1$-preduals with the Daugavet property.

\section*{Acknowledgements}  The author is deeply grateful to the anonymus referee for many helpful suggestions that improved significantly the readability of this paper. The author also thanks A. Avil\'es for fruitful conversations on the topic of the paper. He also thanks J. Langemets and V. Lima for many comments that improved the exposition.


\begin{thebibliography}{999999}

\bibitem{aln13} T. A. Abrahansen,  V. Lima and O. Nygaard,  {\it Remarks on
diameter two properties}, J. Convex Anal. {\bf 20} (2013), (2), 439--452.

\bibitem{alka} F. Albiac, N.J. Kalton, \textit{Topics in Banach Space Theory}, Springer Inc. (2006).

\bibitem {amcrz2022} A. Avil\'es, G. Mart\'inez-Cervantes and A. Rueda Zoca, \textit{$L$-orthogonal elements and $L$-orthogonal sequences}, to appear in Int. Math. Res. Not.

\bibitem {becemartin} J. Becerra and M. Mart\'in, \textit{The Daugavet property for Lindenstrauss spaces}, In In: Jesus,
M.F.C., William, B.J. (eds) Methods in Banach Space Theory, London Mathematical Society
Lecture Note Series, vol. 337 (2006), 91--96.

\bibitem{blrjfa14}
J.~Becerra~Guerrero, G.~L{\'o}pez-P{\'e}rez, and A.~Rueda~Zoca,
  \emph{Octahedral norms and convex combination of slices in {B}anach spaces},
  J. Funct. Anal. \textbf{266} (2014), no.~4, 2424--2435.

\bibitem {blrope} J.~Becerra Guerrero, G.~L\'opez-P\'erez and A.~Rueda Zoca, \textit{Octahedral norms in spaces of operators}, J. Math. Anal. Appl. \textbf{427}, 1 (2015), 171--184.

\bibitem {cll2021} S. Ciaci, J. Langemets and A. Lissitsin, \textit{A characterization of Banach spaces containing $\ell_1(\kappa)$  via ball-covering properties}, to appear in Isr. J. Math.

\bibitem {checos} M. Fabian, P. Habala, P. H\'ajek, V. Montesinos, J. Pelant, V. Zizler, \textit{Functional Analysis and Infinite dimensional Geometry}, CMS Books in Mathematics, Springer-Verlag, New York, 2001.

\bibitem {god} G. Godefroy, {\it Metric characterization of first Baire class
linear forms and octahedral norms,} Studia Math. {\bf 95}, 1 (1989),
1-15.

\bibitem {gk} G. Godefroy and N.~J.~Kalton, \textit{The ball topology and its applications}, Contemporary
Math. \textbf{85} (1989), 195--238.

\bibitem {hln2018} R.~Haller, J.~Langemets and R.~Nadel, \textit{Stability of average roughness, octahedrality, and strong diameter 2 properties of Banach spaces with respect to absolute sums
}, Banach J. Math. Anal. \textbf{12}, 1 (2018), 222--239.

\bibitem {hlp} R.~Haller, J.~Langemets, and M.~P\~oldvere, \emph{On duality of diameter 2 properties}, J. Conv. Anal. \textbf{22} (2015), no.~2, 465--483.

\bibitem {hlp2} R. Haller, J. Langemets, and M. P\~oldvere, \textit{Rough norms in spaces of operators}, Math. Nachr. \textbf{290}, nº 11--12 (2017), 1678--1688.

\bibitem {hww} P.~Harmand, D.~Werner and W.~Werner, \emph{$M$-ideals in Banach
spaces  and Banach  algebras}, Lecture  Notes  in Math.~1547, Springer-Verlag, Berlin-Heidelberg, 1993.

\bibitem {hardtke} J. D. Hardtke, \textit{Summands in locally almost square and locally octahedral spaces}, Acta Comment. Univ. Tartu. Math. \textbf{22}, 1 (2018), 149--162.

\bibitem {llr2} J.~Langemets, V.~Lima and A.~Rueda Zoca, \textit{Octahedral norms in tensor products of Banach spaces}, Q. J. Math. \textbf{68}, 4 (2017), 1247--1260.

\bibitem {ll2019}  J. Langemets and G. L\'opez-P\'erez, \textit{Bidual octahedral renormings and strong regularity in Banach spaces}, J. Inst. Math. Jussieu \textbf{20}, 2 (2021), 569--585.

\bibitem {laru2020} J. Langemets and A. Rueda Zoca, \textit{Octahedral norms in duals and biduals of Lipschitz-free spaces}, J. Funct. Anal. \textbf{279} (2020), article 108557.

\bibitem {linds} J. Lindenstrauss, \textit{Extension of compact operators}, Mem. Amer. Math. Soc. \textbf{48} (1964).


\bibitem {loru2021} G. L\'opez-P\'erez and A. Rueda Zoca, \textit{$L$-orthogonality, octahedrality and Daugavet property in Banach spaces}, Adv. Math. \textbf{383} (2021).

\bibitem {mr2022} M. Mart\'in and A. Rueda Zoca, \textit{Daugavet property in projective symmetric tensor products of Banach spaces}, Banach J. Math. Anal. \textbf{16} (2022), article 35.

\bibitem {werner96} D. Werner, \textit{The Daugavet equation for operators on function spaces}, J. Funct. Anal. \textbf{143} (1997), 117--128.

\end{thebibliography}
\end{document}